\documentclass[12pt]{amsart}

\usepackage[T2A]{fontenc}
\usepackage[utf8]{inputenc}
\usepackage[english]{babel}

\usepackage{amsmath}
\usepackage{amssymb}
\usepackage{amscd}
\usepackage{amsfonts,amsmath,amsthm}

\usepackage[matrix,arrow,curve]{xy}
\usepackage{xfrac}
\usepackage{color}

\newtheorem{thm}{Theorem}
\newtheorem{Ex}{Example}

\newtheorem{prop}{Proposition}

\newtheorem*{main lemma}{The main lemma}

\newtheorem{definition}{Definition}

\title{Diagonal complexes}

\author{Joseph Gordon, Gaiane Panina}

\address{  J. Gordon: Mathematics and Mechanics Faculty, St. Petersburg State University,  joseph-gordon@yandex.ru; \ \ G. Panina: St. Petersburg department of Steklov institute of mathematics, Mathematics and Mechanics Faculty, St. Petersburg State University
 gaiane-panina@rambler.ru}
\keywords{Moduli space, ribbon graphs, curve complex,  associahedron, Chern class. \ \   MSC  52B70, 32G15\ \ UDK 515.164.2 }

\begin{document}

\begin{abstract} {Given an $n$-gon, the poset of all  collections of pairwise non-crossing diagonals
 is isomorphic to the face poset of some convex  polytope  called \textit{associahedron}. We  replace in this setting the $n$-gon (viewed as a disc with $n$ marked points on the boundary) with an arbitrary oriented surface with a number of labeled marked points ("vertices").
With appropriate definitions we arrive at  cell complexes   $\mathcal{D}$ (generalization of)  and its barycentric subdivision $\mathcal{BD}$.  The complex $\mathcal{D}$  generalizes  the associahedron.}
 If the surface is closed, the complex $\mathcal{D}$ (as well as $\mathcal{BD}$) is homotopy equivalent to  the space of metric ribbon graphs $RG_{g,n}^{met}$, or, equivalently, to the decorated moduli space $\widetilde{\mathcal{M}}_{g,n}$.  For bordered surfaces, we prove the following: (1) Contraction of a boundary  edge  does not change the homotopy type of the support of the complex. (2)
Contraction of a boundary component to a new marked point yields a forgetful map between two diagonal complexes which is homotopy equivalent to the  Kontsevich's tautological circle bundle. Thus, contraction of a boundary component gives a natural simplicial model for the  tautological  bundle. As an application, we compute the  \textit{psi-class}, that is,  the first Chern class in combinatorial terms. The latter result is an application of the local combinatorial formula. (3) In the same way, contraction of several boundary components corresponds to Whitney sum of the tautological bundles.

\end{abstract}

\maketitle

\section{Introduction}\label{SectIntro}

We introduce and study  complexes of pairwise non-intersecting curves on an oriented surface (called \textit{diagonals}). Their endpoints belong (by definition) to some fixed set of  labeled marked points (called \textit{vertices}). On the one hand, the complexes generalize the \textit{ associahedron} (or \textit{Stasheff  polytope}). On the other hand, they are directly related to the spaces of metric ribbon graphs (and therefore, to the moduli spaces of punctured algebraic curves), and total spaces of Kontzevich's tautological circle bundles.

\subsection*{Associahedron, \cite{Sta} }  Assume that $n>2$ is fixed. Two diagonals in a convex $n$-gon are \textit{non-intersecting} if they  intersect  only at their endpoints (or do not intersect at all). The set of all  collections of pairwise non-intersecting diagonals in the $n$-gon  is partially ordered by reverse inclusion.
It was shown by John Milnor that the poset is isomorphic to the face poset of some convex $(n-3)$-dimensional polytope $As_n$ called \textit{associahedron}.

In particular, the vertices of the associahedron $As_n$ correspond to the triangulations of the $n$-gon; the edges correspond to edge flips in which one of the diagonals is removed  and replaced by a (uniquely defined) different diagonal.
Single diagonals  are in a bijection with facets of $As_n$, and  the empty set corresponds to the entire  $As_n$.

There exist many  explicit constructions of the associahedron: as a special instance of secondary polytope, truncation of simplex, etc.

There exist also many ways to meaningfully generalize the associahedron. In the present paper, following \cite{Harer} and \cite{Akhmedov},   we consider one more way of generalization.

\subsection*{Metric ribbon graphs, \cite{MP}}
  A\textit{ ribbon graph} is a connected graph (possibly with multiple edges and loops)  together with
a cyclic ordering on the set of germs of edges incident to each vertex. Besides, we assume that each vertex of a ribbon graph has at least three emanating  germs of edges.  A ribbon graph yields an oriented surface, whose genus is called the genus of the graph.   A ribbon graph $\Gamma$ becomes a metric ribbon graph after attaching a positive number $l_i$  to each of its edges $d_1,...,d_d$. Thus isomorphic classes of ribbon graphs label the cells
of the space of metric  ribbon  graphs $RG^{met}_{g,n}$.
It is known due to Harer, Mumford, Thurston, and  Penner that  the space $RG^{met}_{g,n}$ of metric  ribbon  graphs with $n$ faces and genus $g$ can be identified with the decorated moduli space of complex curves of genus $g$ with $n$ distinct labeled marked points. The latter equals the product of the moduli space with the positive cone: $\widetilde{\mathcal{M}}_{g,n}=\mathcal{M}_{g,n}\times \mathbb{R}^n_{>0}$.

By definition (see \cite{Kon}), the tautological  complex line bundle on $\mathcal{M}_{g,n}$ has the cotangent space $T_{v_i}^*C$ at the marked point $v_i$ as the fiber over $(C,v_1,...,v_n) \in \mathcal{M}_{g,n}$. The associated circle bundle\footnote{A necessary reminder:  by a circle bundle we mean a bundle whose fiber is an oriented circle; (isomorphic classes of) complex line bundles correspond bijectively to circle bundles.} on $RG^{met}_{g,n}$ has the $i$-th boundary component of the graph (considered as a metric circle) as the fiber over a point  $(\Gamma, l_1,...,l_d)\in RG^{met}_{g,n}$.

The \textit{Chern classes  }  of the tautological bundles and their products are of a particular interest,  see \cite{Kon},\cite{LanZvo} for a detailed discussion and for expression of the Chern class as a differential $2$-form. In the present paper we apply  N.Mnev and G. Sharygin's local combinatorial formula   and explicitly represent the Chern classes  as  cochains.

\subsection*{Curve  complexes}  Curve complexes (or arc complexes) exist in the literature in different frameworks and settings, see \cite{Thur} and \cite{Harvey} for pioneer papers.    Oversimplifying, the basic idea is to take a (possibly bordered) surface with a finite set of labeled distinguished points, and to associate a complex with the ground set (that is, the set of vertices) equal to  homotopy classes of either closed curves, or (depending on the setting) curves with endpoints in the distinguished set. Simplices correspond to non-intersecting representatives of the homotopy classes.
The  mapping class group has a natural subgroup acting  on the complex, so it makes sense to take the quotient space.

An interesting part of the quotient complex corresponds to collections of curves that cut the surface into a number of disks. The latter is the subject of the present paper.

The existing literature on curve complexes is quite large, since the latter proved to be related to different areas: cluster algebras, low-dimensional dynamical systems, Teichmuller spaces, moduli  spaces of punctured complex curves, measured foliations, and many others.
 Throughout the paper we mention  J.L. Harer's paper \cite{Harer}, where the subject of the paper (the diagonal complex) together with barycentric subdivision appears for the first time.  We also mention R.C. Penner's paper \cite{Pen1} with very similar construction, where he classifies all the cases when the complex is sphere homeomorphic,  and N. Ivanov's survey \cite{Ivanov} with an extension of Thurston's original ideas.
{\textit{Weighted arc families} (in the present paper they correspond to metric diagonal arrangements) appear in the literature diversely, in particular, in relation with measured foliations, e.g. \cite{KaufPenn}. }

\subsection*{Discrete Morse theory}   Discrete Morse theory (developed by R. Forman  \cite{Forman1}, \cite{Forman2}) is a useful technical tool to be used in the paper.
Assume we have a regular cell complex.  A \textit{discrete Morse function}  is an acyclic matching on the Hasse diagram of the complex.
It gives a way of contracting all the cells of the complex that are matched:  if a cell $\sigma$ is matched with its facet\footnote{that is, a cell of dimension  $dim(\sigma)-1$ lying on the boundary of $\sigma$.}  $\sigma'$, then these two can be contracted by pushing $\sigma'$  inside $\sigma$. Acyclicity  guarantees that if we have many matchings at a time, one can consequently perform the contractions. The order of contractions does not matter, and one arrives at a complex homotopy equivalent to the initial one.

\subsection*{Main results of the paper} To single out the complexes that are studied in the present paper, we call them \textit{diagonal complexes}.

With a surface $F$ with a number of labeled marked points, we associate two (explicitly constructed) cell complexes: the complex   $\mathcal{D}$ and  its barycentric subdivision  $\mathcal{BD}$  (Section  \ref{SectMainConstr}).  For closed surfaces, these complexes appeared in a slight disguise in J.L. Harer's paper \cite{Harer}; however, for the sake of the completeness, we present here our construction which is appropriate for consequent paragraphs.

If the surface $F$ has no boundary, and under some other  condition of stability,  $\mathcal{BD}$ (as well as  $\mathcal{D}$) is homotopy equivalent to $RG^{met}_{g,n}$.
Moreover, in this case $\mathcal{BD}$ is a subcomplex of barycentric subdivision of $RG^{met}_{g,n}$  (Section  \ref{SecRelation}).  This result is also contained in \cite{Harer}.

 In the present paper we prove the following:
\begin{enumerate}
  \item The homotopy type of $\mathcal{BD}$ (as well as  $\mathcal{D}$) does not depend on the number $n_i$ of points on a boundary component, provided that $n_i>0$ (Section  \ref{SecContrEdge}).
  \item Contraction of  a boundary component $B_i$ to a new marked point induces a natural forgetful map $\mathcal{BD}({F})\rightarrow \mathcal{BD}(\overline{F})$  which is shown to be isomorphic to the tautological $S^1$-bundle $L_i$, where
 $\overline{F}$ is the surface $F$ with contracted $B_i$ (Section \ref{SecContrBound}). If $\overline{F}$ is a closed surface, the tautological bundle is the Kontsevich's tautological bundle studied in \cite{Kon}.
 As an application, we compute the powers of the first Chern class of the tautological circle bundle in combinatorial terms.
  \item
 Contraction of several boundary components corresponds to Whitney sum of the tautological bundles.
\end{enumerate}

 Summarizing (1), (2), and (3), we have  a complete characterization of the homotopy type of the complex  $\mathcal{BD}$  in terms of $RG^{met}_{g,n}$
 and Whitney sums of tautological bundles.

\subsection*{Acknowledgement} This research is
 supported by the Russian Science Foundation under grant 16-11-10039.

 We are also indebted to Peter Zograf and Max Karev for useful remarks.

\section{Main construction and introductory examples}\label{SectMainConstr}

Assume that we have an oriented surface $F$ of genus $g$ with $b$ labeled boundary components $B_1,...,B_b$.  We fix  $n$ distinct labeled points on $F$ not lying on the boundary. Besides, for each $i=1,..,b$  we fix $n_i>0$ distinct labeled points on  the  boundary component $B_i$.  We assume that $F$ can be triangulated with vertices at the marked points. That is, we exclude all ''small'' cases (like sphere with two marked points).

 Altogether we have $N=n+\sum_{i=1}^b n_i$ marked points; let us call them  \textit{vertices} of $F$. The vertices not lying on the boundary are called \textit{free vertices}. The vertices that lie on the boundary split the boundary components into \textit{edges}.

A \textit{pure diffeomorphism}  $F \rightarrow F$ is an orientation preserving  diffeomorphism which maps fixed points
to fixed points and preserves the labeling.  Therefore, a pure diffeomorphism
maps each boundary component to itself. The \textit{pure mapping class group} $PMC(F)$ is the group of isotopy classes of pure diffeomorphisms.

A \textit{diagonal}  is a simple (that is, not self-intersecting) smooth curve  $d$  on $F$ whose endpoints are some of the (possibly the same) vertices  such that\begin{enumerate}
 \item $d$ contains no vertices (except for the endpoints).
                     \item $d$ does not intersect the boundary (except for its endpoints),
                      \item $d$ is not  homotopic to an edge of the boundary.

                      Here and in the sequel, we mean homotopy with fixed endpoints in the complement  of the vertices $F \setminus \  Vert$.  In other words, a homotopy never hits a vertex.
                     \item $d$ is non-contractible.

                   \end{enumerate}

An \textit{admissible diagonal arrangement} (or an \textit{admissible arrangement}, for short) is a non-empty collection of diagonals $\{d_j\}$ with the properties:

\begin{enumerate}
\item Each free vertex is an endpoint of some diagonal.
  \item No two diagonals intersect (except for their endpoints).
  \item No two diagonals are homotopic.
  \item The complement of the arrangement and the boundary components $(F \setminus \bigcup d_j) \setminus \bigcup B_i$  is a disjoint union of open disks.
\end{enumerate}

\begin{definition}\label{DefEquivArr}
Two arrangements  $A_1$ and $A_2$ are \textit{strongly equivalent}  whenever there exists a homotopy taking $A_1$ to $A_2$.

Two arrangements  $A_1$ and $A_2$ are \textit{weakly equivalent}  whenever there exists a pure diffeomorphism of $F$ which maps bijectively $A_1$ to $A_2$.
\end{definition}

\medskip

\textbf{Remark.}  If there are no boundary components, weak equivalence classes of admissible arrangements correspond bijectively to ribbon graphs from $RG_{g,n}$.

\medskip

Arrangements with maximal possible number of diagonals  $Max=6g+3b+2n+N-6$  correspond to triangulations of $F$ with vertices at fixed points. Here  by triangulation we mean that the disks of the complement are combinatorial triangles, but possibly self-intersecting on the boundary.  Arrangements with minimal possible number of diagonals $Min= 2g+n+b-1$  have a unique disc in the complement.

A triple $(g,b,n)$ is \textit{stable
} if no admissible arrangement has  a non-trivial automorphism (that is, each pure diffeomorphism which maps an arrangement to itself, maps each germ of each of $d_i$ to itself). Triples with $b>1$ are stable since a boundary component allows to set a linear ordering on the germs of diagonals emanating from each of its vertices.
It is known\footnote{This follows from Lefschetz fixed point theorem, as explained  by Bruno Joyal in personal communications.} that any triple with $n>2g+2$ is stable.

Throughout the paper we assume that \textbf{all  the triples are stable}.

\medskip

\subsection*{Poset $\widetilde{D}$ and cell complex $\widetilde{\mathcal{D}}$.}

Strong equivalence classes of admissible arrangements are partially ordered by reversed inclusion: we say that $A_1 \leq A_2$ if there exists a homotopy that takes the arrangement $A_2$ to some subarrangement of $A_1$.

Thus  for the data  $ (g, b , n; n_1,...,n_b)$ we have the poset of all strong equivalence classes of admissible arrangements  $\widetilde{D}={\widetilde{D}}_{g, b , n; n_1, \ldots,n_b}$.

\begin{Ex}\label{ExAssoc}  The poset   ${\widetilde{D}}_{0, 1,0; n_1}$  is isomorphic to the face poset of the associahedron $As_{n_1}$.
In this case any collection of pairwise non-intersecting diagonals is admissible.
 \end{Ex}

In view of this example we are going to generalize associahedron. The surface $F$ plays the role of the polygon, marked points play the role of  vertices.

The poset $\widetilde{D}$ can be realized as the poset of some (uniquely defined)  regular\footnote{
A cell complex $K$ is \textit{regular} if each  $k$-dimensional cell $c$ is attached to some  subcomplex of
the $(k-1)$-skeleton of $K$ via a bijective mapping on $\partial c$.} cell complex $\widetilde{\mathcal{D}}$.
Indeed, let us build up  $\widetilde{\mathcal{D}}$  starting from the cells of maximal dimension. Each such cell corresponds to cutting of the surface $F$ into a single polygon. Adding more diagonals reduces the general case to Example \ref{ExAssoc}. In other words,   $\widetilde{\mathcal{D}}$ is a patch of associahedra.

 For the most examples, $\widetilde{\mathcal{D}}$ has infinitely many cells. Our goal is to factorize $\widetilde{\mathcal{D}}$ by the action of the pure mapping class group. For this purpose consider the defined below barycentric subdivision of $\widetilde{\mathcal{D}}$.

\subsection*{Poset $\widetilde{BD}$ and cell complex $\widetilde{\mathcal{BD}}$.} We apply now the construction of the \textit{order complex} \cite{Wachs} of a poset, which gives us barycentric subdivision.
  Each  element of the  poset ${\widetilde{BD}}_{g, b , n; n_1,... ,n_b}$ is (the strong equivalence class of) some admissible arrangement $A=\{d_1,...,d_m\} $ with a linearly ordered partition $A=\bigsqcup S_i$   into some non-empty sets  $S_i$  such  that the first set $S_1$ in the partition  is an admissible arrangement.

The partial order on $\widetilde{BD}$ is generated by the following rule:
 \newline $(S_1,...,S_p)\leq (S'_1,...,S'_{p'})$ whenever one of the two conditions holds:
\begin{enumerate}
  \item We have one and the same arrangement  $A$, and $(S'_1,...,S'_{p'})$ is an order preserving refinement of $(S_1,...,S_p)$.
  \item $p\leq p'$, and for all $i=1,2,...,p$, we have  $S_i=S'_i$. That is, $(S_1,...,S_p)$ is obtained from $(S'_1,...,S'_{p'})$ by removal
  $S'_{p+1},...,S'_{p'}$.
\end{enumerate}

Let us look at the incidence rules in more details. Given $(S_1,...,S_p)$, to list all the elements of $\widetilde{BD}$ that are smaller than $(S_1,...,S_p)$  one has
(1) to eliminate some (but not all!) of $S_i$ from the end of the string, and (2) to replace some consecutive collections of sets by their unions.

\medskip

Examples:$$(\{d_5,d_2\},\{d_3\},\{d_1,d_6\},\{d_4\},\{d_7\},\{d_8\})\ \  >\ \ (\{d_5,d_2\},\{d_3,d_1,d_6\},\{d_4,d_7\}).$$
$$(\{d_5,d_2\},\{d_3\},\{d_1,d_6\},\{d_4\},\{d_7\},\{d_8\})\ \  >\ \ (\{d_5,d_2\},\{d_3\},\{d_1,d_6\},\{d_4\},\{d_7\}).$$
$$(\{d_5,d_2\},\{d_3\},\{d_1,d_6\},\{d_4\},\{d_7\},\{d_8\})\ \  >\ \ (\{d_5,d_2\},\{d_3\},\{d_1,d_6\},\{d_4\},\{d_7,d_8\}).$$

Minimal elements of $\widetilde{BD}$ correspond to admissible arrangements. Maximal elements correspond to maximal arrangments  $A$ together with some minimal admissible subarrangement $A'\subset A$ and a linear ordering on the set $A\setminus A'$. For maximal elements, the number of sets in the partition  $p=Max-Min+1$.

By construction, the complex  $\widetilde{\mathcal{BD}}$  is combinatorialy isomorphic to the barycentric subdivision of $\widetilde{\mathcal{\mathcal{D}}}$.

We are mainly interested in the quotient complex:

\begin{definition}
  For a fixed data $(g, b , n; n_1,... ,n_b)$, the diagonal  complex $\mathcal{BD}_{g, b , n; n_1,... ,n_b}$ is defined as

  $$\mathcal{BD}=\mathcal{BD}_{g, b , n; n_1,... ,n_b}:=\mathcal{\widetilde{BD}}_{g, b , n; n_1,... ,n_b}/PMC(F).$$
  We define also
   $$\mathcal{D}=\mathcal{{D}}_{g, b , n; n_1,... ,n_b}:=\mathcal{\widetilde{D}}_{g, b , n; n_1,... ,n_b}/PMC(F).$$
\end{definition}
 Alternative definition reads as:

 \begin{definition}
   Each cell of the complex ${{\mathcal{BD}}}_{g, b , n; n_1,... ,n_b}$ is labeled by the weak equivalence class of some admissible arrangement $A=\{d_1,...,d_m\} $ with a linearly ordered partition $A=\bigsqcup S_i$   into some non-empty sets  $S_i$  such  that the first set $S_1$ in the partition  is an admissible arrangement.

The incidence rules  are the same as the above rules for the complex $\widetilde{\mathcal{BD}}$.
 \end{definition}

\begin{prop}  The cell complex $\mathcal{BD}$ is regular. Its cells are combinatorial simplices.

\end{prop}
Proof.   If $(S_1,...,S_r)\leq (S'_1,...,S'_{r'})$ then there exists a unique (up to isotopy)  order-preserving  pure diffeomorphism of $F$ which  embeds \newline  $A= S_1\cup ...\cup S_r$ in  $A'=S'_1\cup ...\cup S'_{r'}$.
Indeed,   If $S_1=S_1'$,  the arrangement $S_1$  maps identically to itself since it has no automorphisms by stability assumption. The rest of the diagonals are diagonals in polygons, and are uniquely defined by their endpoints. Assume that $S_1 \subset S_1'$.
For the rest of the cases it suffices to take $A=S_1$, $A'=A =S_1'\bigsqcup S_2'$. If $A$ embeds in $A'$ in different ways,
then $A$ has a non-trivial isomorphism, which contradicts stability assumption.
\qed

\medskip
 The complex $\mathcal{D}=\mathcal{\widetilde{D}}_{g, b , n; n_1,... ,n_b}/PMC(F)$ is usually non-regular, see the below examples.

 \begin{Ex}\label{ExSimplest}\begin{enumerate}
$\mathcal{D}_{0,2,0;1,1}$ is  a combinatorial circle. It has one vertex and one edge. $\mathcal{BD}_{0,2,0;1,1}$, which is also a combinatorial circle, has two vertices and two edges, see Fig \ref{Fig1}.
           \end{enumerate}
 \end{Ex}

\begin{Ex}\label{ExCylinder}\begin{enumerate}
$\mathcal{D}_{0,2,0;1,2}$ is  the cylinder $I\times S^1$. It has four vertices, six edges, and two pentagonal cells, see Fig \ref{Fig2}.  Each of the pentagonal cells patches to itself by an edge.
           \end{enumerate}
 \end{Ex}

\begin{figure}[h]
\centering \includegraphics[width=8 cm]{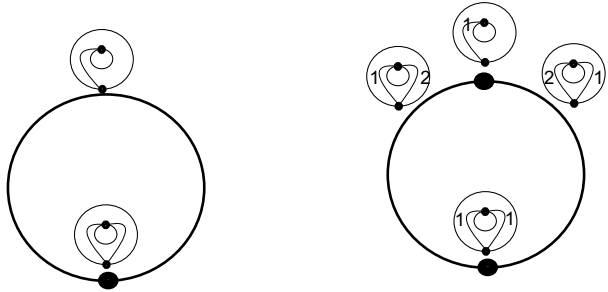}
\caption{Complexes $\mathcal{D}_{0,2,0;1,1}$ and  $\mathcal{BD}_{0,2,0;1,1}$}\label{Fig1}
\end{figure}

\begin{figure}[h]
\centering \includegraphics[width=8 cm]{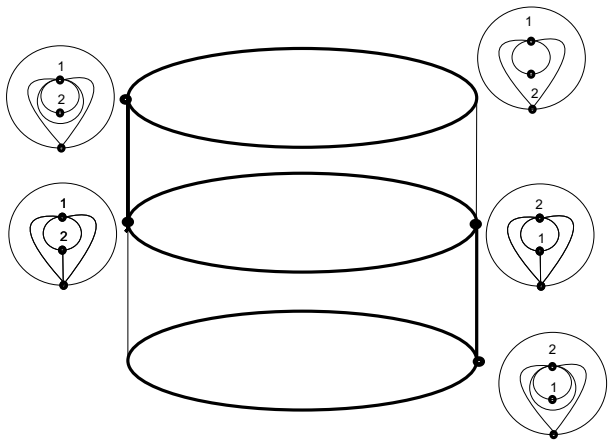}
\caption{Complex $\mathcal{D}_{0,2,0;2,1}$. We depict labels of all the vertices and one of the edges. This figure (without labels) appeared in \cite{Akhmedov}}\label{Fig2}
\end{figure}

\medskip

\medskip

\section{ Relation to $RG^{met}_{g,n}$. Attaching lengths.}\label{SecRelation}
\textbf{Starting from now, by an admissible arrangement we mean the weak equivalence class of an admissible arrangement.}

Admissible arrangements bijectively correspond to ribbon graphs  by graph duality. Contracting an edge in a ribbon graph corresponds to eliminating the dual diagonal from the dual arrangement.

\begin{thm} \begin{enumerate}
              \item $RG^{met}_{g,n}$ (considered as a topological space) is homotopy equivalent to $\mathcal{BD}_{g,o,n}$ (and therefore also to $\mathcal{D}_{g,o,n}$).
              \item $\mathcal{BD}_{g,o,n}$ embeds in (the analog of) barycentric subdivision of $RG^{met}_{g,n}$ as a deformation retract.
            \end{enumerate}

\end{thm}
Comments on the proof.
The theorem (in a slight disguise) was proven in \cite{Harer}. An independent proof is contained in \cite{DunBarkSh}. {This construction appears also in K. Igusa's category of fat graphs \cite{Igusa}.}
 The rough idea of the proof is to compare  $\mathcal{BD}_{g,o,n}$ with the barycentric subdivision of $RG^{met}_{g,n}$. Strictly speaking, $RG^{met}_{g,n}$ is not a cell complex in the sense of A. Hatcher's book \cite{Hatcher}, although it is
patched of the open balls. Each of the balls   correspond to some admissible arrangement of $m$ diagonals, see \cite{MP} for details.
However, it has a  well-defined barycentric subdivision  containing a natural embedding of the complex  $\mathcal{BD}_{g,o,n}$. \qed

\medskip

In view of the above construction, it is possible to attach lengths  to  diagonals  for any complex $\mathcal{BD}$, even when $F$ has boundary components.
This gives a \textit{metric arrangement}.
Namely, we have:
\begin{thm} \label{ThmLengths} The  support of $\mathcal{BD}_{g, b , n; n_1,... ,n_b}$ equals the space of
admissible arrangements equipped  with  a length function $$l:\{d_i\} \rightarrow \mathbb{R}_{>0},$$
which satisfies the two conditions:
\begin{enumerate}
  \item For each  $(A,l)$ the length function $l$ attains its maximal value on some admissible $A' \subseteq A$.
  \item  $\sum_{d_i \in A} l(d_i)=1$.
\end{enumerate}
Vanishing of $l(d_i)$ means eliminating the diagonal $d_i$. \qed
\end{thm}

Note that in our setting, no length is attached to the edges of $F$.

\begin{definition}\label{DefTaut} Let $v_i$ be a free vertex. For a metric diagonal arrangement let $l_1,l_2,...,l_m$ be the lengths
 of diagonals emanating from $v_i$ \footnote{one and the same diagonal may appear twice.} coming in the counter clockwise order. The tautological circle bundle $L_i$  on $\mathcal{BD}$ is the bundle  whose fiber over
a metric arrangement is the (combinatorial) polygon with consecutive edge lengths $l_1,l_2,...,l_m$.
\end{definition}

If there are no boundary components, $L_i$ equals the tautological bundle introduced in \cite{Kon}.

\medskip

\textbf{Remark.} One can  relax the condition  $\sum l_i=1$ and define the length assignment to an arrangement as a point in the real projective space.  This will be convenient in the subsequent sections where we'll eliminate some of diagonals.

\medskip

\textbf{Remark.}
Although Theorem \ref{ThmLengths}  represents each simplex of the complex $\mathcal{BD}$ as some metric simplex,
for the consequent paragraphs a reader may imagine
 each (combinatorial) simplex in ${\mathcal{BD}}_{g, b , n; n_1,... ,n_b}$ as a (Euclidean) equilateral simplex and to
define the \textit{support}, or \textit{geometric realization} of the complex $| {\mathcal{BD}}|_{g, b , n; n_1,... ,n_b}= |{\mathcal{D}}|_{g, b , n; n_1,... ,n_b}$  as the patch of these simplices.

\section{Contraction of edges} \label{SecContrEdge}

\begin{thm} Homotopy type of the support \(|{\mathcal{D}}|_{g, b , n; n_1,... ,n_b}\ =\ |{\mathcal{BD}}|_{g, b , n; n_1,... ,n_b}\) depends only on the triple \((g,b,n)\).
\end{thm}

\begin{proof}
Prove that  \(|{\mathcal{BD}}|_{g, b , n; n_1-1,... ,n_b}\) is
 a deformation retract of \(|{\mathcal{BD}}|_{g, b , n; n_1,... ,n_b}\) provided that
\(n_1>1\).

 Choose an edge \(e\) with endpoints \(v',v''\) (taken in counter-clockwise order) on the boundary component \(B_1\). Consider the following forgetful poset epimorphism (the latter depends on the chosen edge).

\[
\pi: BD_{g, b , n; n_1,... ,n_b}\rightarrow BD_{g, b , n; n_1-1,... ,n_b}
\]

The defining rule is as follows.
An element of \(BD_{g, b , n; n_1,... ,n_b}\)  gives us some admissible arrangement together with a partition
\((S_1,...,S_r)\). Contract the edge \(e\) to  a  new  vertex \(v\), which replaces the former vertices \(v',v''\). We obtain a (new) collection of diagonals related to the surface with contracted edge. Some of the diagonals may become either contractible or homotopic to an edge of \(F\). Eliminate them.
Some of the diagonals may become pairwise homotopy equivalent. In each class we leave exactly one that belongs to \(S_i\) with the smallest index \(i\).
Eventually some of the sets \(S_i\) may become empty in the process. Eliminate all the empty sets keeping the order of the rest. We obtain an element from $BD_{g, b , n; n_1-1,... ,n_b}$. It is easy to check that $A<A'$ implies $\pi(A)\leq \pi(A')$, so the map is indeed a poset morphism.

 The poset morphism extends to a   piecewise linear map (we denote it by the same letter $\pi$):

\[
\pi:| \mathcal{BD}|_{g, b , n; n_1,... ,n_b}\rightarrow | \mathcal{BD}|_{g, b , n; n_1-1,... ,n_b}
\]
which is linear on each of the simplices.

The preimage of each point carries the structure of a regular cell complex; let us
show that it is a combinatorial segment.

\begin{figure}[h]
\centering \includegraphics[width=14 cm]{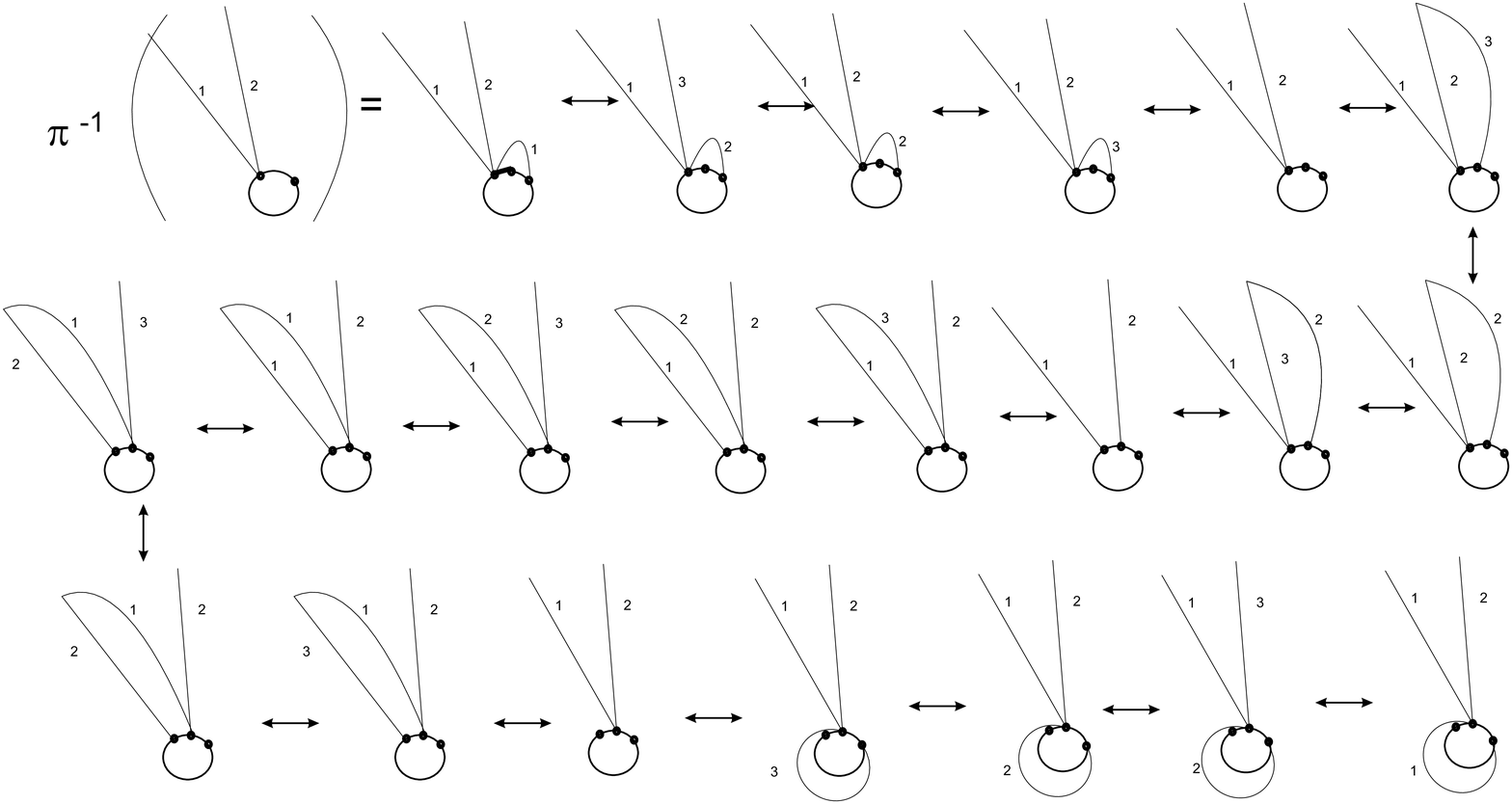}
\caption{The preimage of  a configuration is a combinatorial segment, and thus is contractible to the lefthand end of the segment. The edge which gets contracted is marked bold.}\label{segment}
\end{figure}

Take an inner  point \(
x \in \sigma^{r-1}\) of  a simplex  labeled by \((S_1,...,S_r)\). Assume that for \(\sigma^{r-1}\), in the corresponding arrangement the vertex \(v\) has \(m\) emanating germs of diagonals and two germs of incident boundary edges \(d_0,d_1,...,d_m,d_{m+1}\), coming in the counterclockwise order. Some of the germs may correspond to one and the same diagonal (or one and the same edge). Each  simplex in the preimage of \(x\)  is obtained in the following way.

 Expand the edge \(e\) either before all the germs, or after all the germs, or between between two consecutive germs.  Now we have two cases:
  \begin{enumerate}
    \item Leave the arrangement as it is (keeping the partition). Then the arrangement corresponds to some one-dimensional simplex in the preimage.
    \item Add a diagonal to the arrangement. Then the arrangement corresponds to a vertex of the preimage. Here we again have two cases:
     \begin{enumerate}
       \item The new diagonal $d'$ becomes homotopic to some $d\in S_i$  after collapsing $e$. Then we put $d'$ in the same set \(S_i\) or in any set to the right of \(S_i\).
 We also can create a separate singleton $\{d'\}$ and put it to the right of $S_i$.
       \item The new diagonal $d'$ is homotopic to a piece of the boundary component  $B_1$. In this case we put $d'$ in any of $S_i$, or create a separate singleton, see Figure \ref{segment}.
     \end{enumerate}
  \end{enumerate}

One can check that altogether we have a segment in the preimage: first check that each vertex is incident to at most two segments, then check the connectivity of the preimage.

All of the simplices in the sequence are evidently distinct (one can easily show it by comparing degrees of vertices and capacities of sets \(S_i\)).  So we indeed have a combinatorial segment.

Take a subcomplex \(\mathcal{K}\subset  \mathcal{BD}_{g, b , n; n_1,... ,n_b}\) labeled by all the arrangements which have a diagonal homotopic to the conjunction of \(e\) and the next edge of boundary in the clockwise order \(e'\) (such a curve is indeed a diagonal) lying in the set \(S_1\). Obviously such arrangement has no germs emanating from \(v'\). An example of such arrangement is the leftmost in the Figure \ref{segment}, that is, $K$ consists of all lefthandside vertices of the segments $\pi^{-1}(x)$, where $x$ ranges over  \(|\mathcal{BD}|_{g, b , n; n_1-1,... ,n_b}\). Observe that $\pi$ maps $|\mathcal{K}|$ isomorphically to \(|\mathcal{BD}|_{g, b , n; n_1-1,... ,n_b}\).

Now we can fiberwisely retract  $|\mathcal{BD}|_{g, b , n; n_1,... ,n_b}$  onto $|\mathcal{K}|$.  To correctly explain the retraction, we make use of discrete Morse theory.
Consider a matching on \({\mathcal{BD}}_{g, b , n; n_1,... ,n_b}\) by pairing in the preimage of each simplex \(\sigma^{r-1}\) each \(r\)-simplex with its neighbor in such a way that the unique non-paired cell in \(\pi^{-1}(\sigma^{r-1})\) lies in \(K\). Clearly, we have an acyclic matching.
\end{proof}

The technique of the section appeared in a disguise  in Penner's paper \cite{Pen1}, using  \textit{train tracks}  technique, where the author studies a similar (still different) complex. A reader familiar with this subject   remembers that adding a point to the boundary component in \cite{Pen1}  amounts  to taking a suspension over the arc complex  whereas in our setting we have a homotopy equivalence. This phenomenon is easy to understand by looking at the very first example  (a disk with marked points on the boundary). Our  example gives a ball (the associahedron, including the interior), whereas Penner's setting gives the boundary complex of a polytope dual to the associahedron.

\section{Contraction of boundary components} \label{SecContrBound}

 Assume that the first boundary component $B_1$ has exactly one marked point. Let us contract it and turn it to a new free vertex $v$ labeled by $n+1$.
 We have a forgetful cellular mapping
$$\pi: \mathcal{BD}_{g, b , n; 1,n_2,n_3,...,n_b} \rightarrow \mathcal{BD}_{g, b-1, n+1; n_2,... ,n_b}$$
whose defining rule is literally  the same as in the previous section, namely: a simplex in  $\mathcal{BD}_{g, b , n; 1,... ,n_b}$ corresponds to some admissible arrangement.
 After contraction  of $B_1$
some of the diagonals may become  contractible. Eliminate them.
Some of the diagonals may become homotopy equivalent. In each homotopy equivalence class we leave exactly one that belongs to $S_i$ with the smallest index $i$.

The mapping induces  a continuous mapping for the supports, which we denote by the same letter $\pi$.

\begin{thm}  \begin{enumerate}
               \item For the contraction of a boundary component with one marked point, the triple $${\pi}: | \mathcal{BD}|_{g, b , n; 1,n_2,n_3,...,n_b}\ \rightarrow \ | \mathcal{BD}|_{g, b-1, n+1; n_2,... ,n_b}$$

is homotopy equivalent  to the tautological circle bundle $L_{n+1}$ over  ${\mathcal{BD}}_{g, b-1, n+1;n_2,n_3,...,n_b}$.
\item In particular,  $${\pi}: | \mathcal{BD}|_{g, 1 , n; 1}\ \rightarrow \ | \mathcal{BD}|_{g,0, n+1;}=\widetilde{\mathcal{M}}_{g,n}$$

is  the Kontsevich's tautological circle bundle $L_{n+1}$ over \newline $\widetilde{\mathcal{M}}_{g,n}$.
               \item For the contraction of a boundary component with several  marked points,  we again have homotopy equivalence with the tautological circle bundle $L_{n+1}$.
             \end{enumerate}
\end{thm}

Proof. Since (1) implies (2) and (3), it suffices to prove (1).

  Take a  simplex $\sigma \in \mathcal{BD}_{g, b-1, n+1; n_2,... ,n_b}$ labeled by
$(S_1,...,S_r)$. Assume that for $\sigma$, in the corresponding arrangement  $v$ has $m$ emanating germs of diagonals $d_1,...,d_m$, coming in the counterclockwise order.  Some of the germs may correspond to one and the same diagonal.
Consider  the  preimage of  a point $x$ lying in  the interior of $\sigma$. The preimage carries the structure of a regular cell complex.  The simplices in the preimage are obtained by the following procedure: place the new boundary component to the vertex $v$. Either leave it as it is,
or add a curve which duplicates one of the two neighbor diagonals  $d\in S_i$.  Put the new diagonal $d'$ either in the same set $S_i$ as $d$, or   to any of the sets with bigger indices, or as a singleton to any place to the right of $S_i$.

\medskip

Important remark: might happen   that the two neighbor germs are the germs of one and the same diagonal.  Then the diagonal can be duplicated in two ways, so this  case does not create an exception in our construction.

\medskip

It is easy to see that ${\pi}^{-1}(x)$ is a combinatorial circle.  Figure  \ref{Circle} depicts the preimage ${\pi}^{-1}(x)$ for the case when the collapsed boundary component has exactly two emanating diagonals, one from $S_1$, and the other from $S_2$.

The generic case is captured by the following observation: each preimage is connected, it carries a structure of one-dimensional cell complex,  each vertex of which has exactly two adjacent edges.

 \medskip

\textbf{Grid on the circle.}
Let us  explicitly describe the combinatorics of the circle in the preimage.
Assume that a point $x$ lies  in a cell  of the base labeled by $(S_1,...,S_r)$.  Assume that the new vertex has $m$ emanating germs of edges
$e_i\in  S_{j_i}$.
Take an oriented circle  and construct the following grid.  \begin{enumerate}
                            \item Put $m$ bold points on the circle. The latter correspond to pairs of neighbor germs of diagonals emanating from $v$.
                            Therefore, each segment  between  two consecutive    points corresponds to a germ.
                            \item For each $i=1,...,m$, put $2(r-j_i)+1   $  points on the corresponding segment.
                          \end{enumerate}

 We look at the circle with points as at a cell complex. In view of the above discussion
 this cell  complex is combinatorially isomorphic to $\pi^{-1}(x)$.  A cell of the grid tells us (1) between which germs the boundary component should be inserted, (2) which emanating edge should be duplicated, and (3) what is the partition on the new set of edges.

\begin{figure}[h]
\centering \includegraphics[width=10 cm]{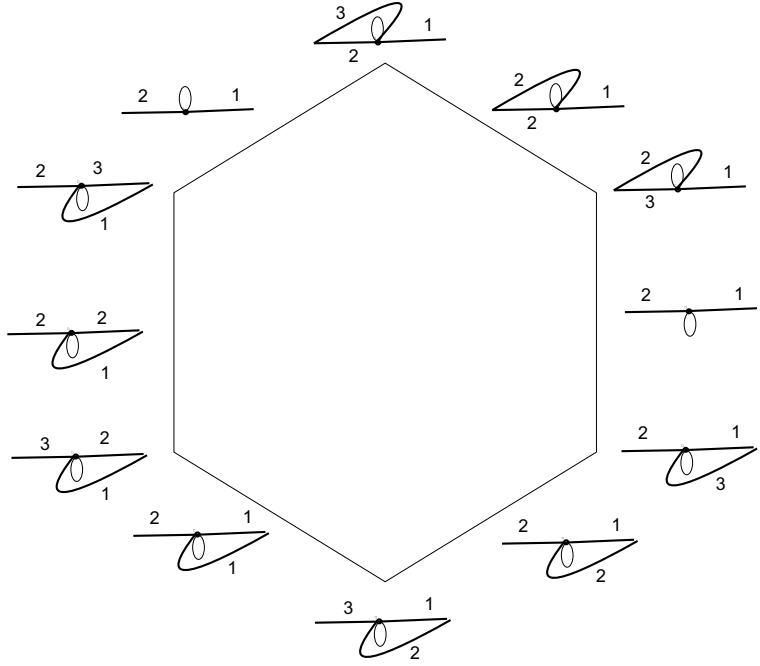}
\caption{The preimage is the combinatorial hexagon.  We depict here the corresponding arrangements locally, near the first boundary component, since the rest of the diagonals remains unchanged.}\label{Circle}
\end{figure}

 We are almost done; however, we need a metric combinatorial circle to ensure that we have the tautological  circle bundle.  Take the circle with the  grid and leave the bold points only.  They cut the circle into segments that bijectively correspond to germs  of diagonals emanating from $v$.  Assign to each of the edges the length of the corresponding diagonal,
 see Fig. \ref{Bundle}.

The behaviour of the metric circle $\pi^{-1}(x)$ when the point $x$ moves on the base is captured by the following observations:
if $x$ stays in one and the same simplex of $\mathcal{BD}$, the combinatorics of the metric circle does not change.
If $x$ meets a face of the simplex which corresponds to a coarser partition of the same diagonal arrangement (that is, no diagonals get removed), then again the combinatorics of the metric circle does not change.  If $x$ meets a face of the simplex which corresponds to a  removal of some diagonals that are not incident to $v$, then again the combinatorics of the circle does not change. Finally, removal of  diagonals that are  incident to $v$  means that corresponding lengths $l(d_i)$ tend to zero, and eventually the corresponding edges of the circle collapse.

\qed

\begin{figure}[h]
\centering \includegraphics[width=8 cm]{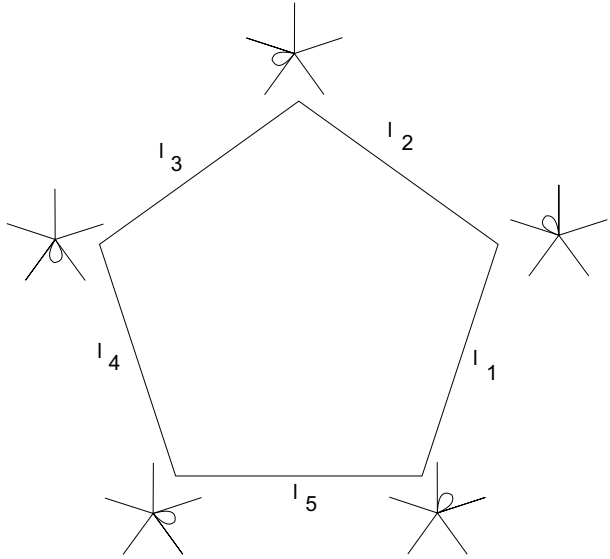}
\caption{The fiber with assigned lengths.}\label{Bundle}
\end{figure}

The above approach is very much related to $(S^1)^b$  action on the arc complex described in \cite{Pen1}. The description uses train tracks  technique.

\medskip

 Assume  now  that we contract  $k$ boundary components  to  $k$ new free vertices. We again have a well-defined forgetful map $$\pi: \mathcal{BD}_{g, b , n; 1,1,1,...,1} \rightarrow \mathcal{BD}_{g, b-k, n+k; 1,... ,1}.$$

\medskip
\begin{thm}  For the contraction of $k$ boundary components at a time, the bundle
$${\pi}: | \mathcal{BD}|_{g, b , n; 1,1,1,...,1}\  \rightarrow \ | \mathcal{BD}|_{g, b-k, n+k; 1,... ,1}$$
  is homotopy equivalent to Whitney sum of the tautological $S^1$-bundles  $L_{n+1}\oplus... \oplus L_{n+k}$  on $\mathcal{BD}_{g, b-k, n+k; 1,... ,1}.$
\end{thm}
Proof.   Assume we have a number of combinatorial $S^1$-bundles over one and the same base complex $\mathcal{BD}$. Then their Whitney sum carries a natural cell structure: each fiber decomposes into products of the cells of the summands. In other words, we obtain a \textit{grid on the torus}  on each of the fibers.  Each cell of such a grid is a (combinatorial) cube.

Let us examine the preimage  $\pi^{-1}(x)$ of a point $x$ lying on the base in a cell labeled by $(S_1,...,S_r)$.
It also carries some combinatorial  structure  to be compared with the grid on the torus.

Assume that the contraction $B_1$,...,$B_k$ gives new free marked points $v_1,...,v_k$.
The preimage can be subdivided in \textit{fragments}:
 each  fragment of the preimage of a point $x\in (S_1,...,S_r)$  corresponds to a choice of germs incident to $v_1,...,v_k$, one germ per a vertex.
 The fragment corresponds to (1) placing boundary components next to the chosen germ (either to the left or to the right),  (2) duplicating some of the diagonals containing the chosen germs  and (3) deciding where to place the new diagonals  in the partition  $S_1,...,S_r$.
 Items (1) and (2) are depicted in Figure \ref{FigFrag}.

\begin{figure}[h]
\centering \includegraphics[width=10 cm]{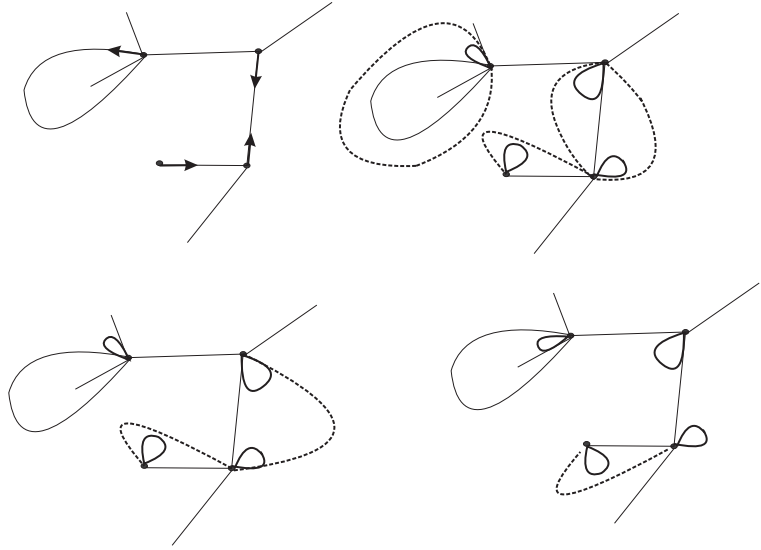}
\caption{The choice of emanating germs determines a fragment of $\pi^{-1}(x)$. We depict here some of the possible placements of boundary components (bold circles) and some of the possible ways of adding new diagonals (bold dashed lines).}\label{FigFrag}
\end{figure}

In each fiber of the Whitney sum $L_{n+1}\oplus... \oplus L_{n+k}$ , the preimage of $x$  is also subdivided in analogous number of fragments constituting the grid on the torus. As we discussed above, the interiors of the fragments are topologically open balls.
We shall show that the interior of each  fragment for $$\pi: \mathcal{BD}_{g, b , n; 1,1,1,...,1} \rightarrow \mathcal{BD}_{g, b-k, n+k; 1,... ,1}$$
is also a topological ball, so it is possible to  identify the fibers of the two bundles.

If the point $x$ moves on the base to the boundary of the cell, the way of degenerating the boundaries of the fragments is one and the same for the Whitney sum $L_{n+1}\oplus... \oplus L_{n+k}$  and for $\pi: \mathcal{BD}_{g, b , n; 1,1,1,...,1} \rightarrow \mathcal{BD}_{g, b-k, n+k; 1,... ,1}$.   Therefore
we have an isomorphism between the bundles.

\medskip
\textbf{Case 1:  contraction of two boundary components, no connecting edges.}

Assume that we contract  two boundary components  to new free marked points $v_1$ and $v_2$.
If no diagonal in $A=\bigcup S_i$ connects $v_1$ and $v_2$, the preimage carries the cell structure of the above described product of the two grids related to $L_{n+1}$ and $L_{n+2}$  with one exception  which we describe below. Cells of the product corresponds to adding new diagonals, say, $d_1$ and $d_2$.  Each of the diagonals is added either in some of $S_i$, or after one of $S_i$.  If the two diagonals are added after one and the same $S_i$, the corresponding two-cell of the product of grids gets partitioned  further.
This additional partition   consists of two triangular cells  $(S_1,...,S_i,\{d_1\}, \{d_2\}, S_{i+1},...,S_r)$  and  $(S_1,...,S_i,\{d_2\},\{d_1\}, S_{i+1},...,S_r)$, and   the diagonal segment
\newline $(S_1,...,S_i,\{d_1,d_2\}, S_{i+1},...,S_r)$.

\medskip
\textbf{Case 2:  contraction of $k$ boundary components, no connecting edges.}

If more than two boundary components are contracted, and no diagonal in $A=\bigcup S_i$ connects $v_i$ and $v_j$, the preimage carries the cell structure which refines the  grid on the torus. Namely, some of the cubes of the product are partitioned further.
It is easy to see that each of the partitions is combinatorially isomorphic to a product of duals to \textit{permutohedra.}\footnote{The permutohedron $P_n$ is a polytope whose face poset is combinatorially isomorphic  to the poset of linearly ordered partition of the set $\{1,...,n\}$.}

\medskip
\textbf{Case 3: two boundary components, one connecting edge.}

 If the set $A$ contains a diagonal connecting   $B_1$ and $B_2$  (say, $d$), existing  ways of duplicating $d$ begin to interfere. Therefore we have another type of exception called  \textit{elementary exceptional fragment}. The latter corresponds  to placing the two new boundary components next to  the connecting diagonal $d$. In this case the Figure \ref{ExFragment}  depicts the combinatorics of the exceptional fragment for the case  $r=1$, that is, for  $A=S_1$. For $r>n$ we have a refinement of  the cell structure depicted in the figure, but in any case  it is a topological disc.  It is important that on its boundary, the exceptional fragment coincides with the grid on the torus.

\begin{figure}[h]
\centering \includegraphics[width=10 cm]{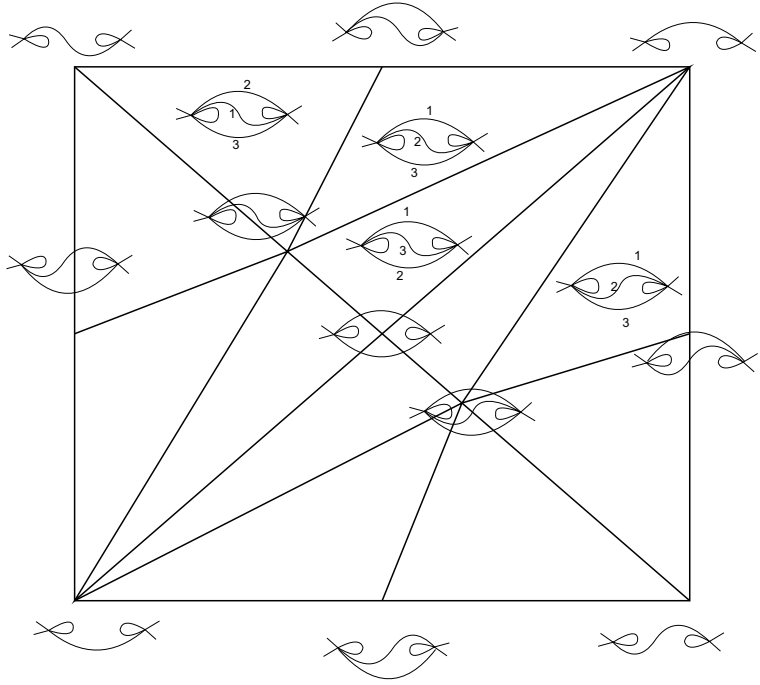}
\caption{Elementary exceptional fragment for the case $r=1$. We depict here all the labels of the vertices (all associated numbers equal one), and four labels of two-cells. The numbering  on the curves indicates the partition. }\label{ExFragment}
\end{figure}

\medskip

\textbf{Example. } It is instructive to look at the following ''limit case'':  one of the vertices (say, $v_1$)  has a unique emanating diagonal  in the set $A$, and this  diagonal leads to $v_2$. Then $v_2$ necessarily has emanating diagonals that  lead to other vertices. It is easy to see that in this case the closure of the exceptional fragment
patches to itself
      by identifying a pair of opposite sides.  However,  the sides are ''non-exceptional'', and we still have a contractible exceptional part.

\medskip

\medskip
 The \textbf{generic case}  ($k$ boundary components, several connecting edges)
 reduce this case to the above described ones.
There are several exceptional fragments. We need to prove that each of them  is an open topological disc.

 Let us fix a choice of a fragment, that is, a choice of emanating germs incident to $v_1,...,v_k$.
 The cells of the fragment correspond to all possible ways of duplication of the chosen diagonals.

 A new diagonal is called \textit{movable}  if it duplicates some $d\in S_i$, but is placed \textbf{not} in $S_i$.
 For a fixed choice of germs and new diagonals, let  us order all the movable diagonals in such a way that
 \begin{enumerate}
   \item First come diagonals that duplicate $d\in S_i$  with smaller values of $i$, and
   \item If a diagonal is duplicated by more than one new movable diagonals, the latter come one after another.
 \end{enumerate}
 Assume that the movable diagonals are $d_1,...,d_m$, and $d_i \in S_{j_i}$.
 Let us set a discrete Morse function on the cells of the fragment:
 \begin{itemize}
   \item Step 1.  We match   $(*,S,\{d_1\},*)$ with $(*,S\cup \{d_1\},*)$  if $S\neq S_{j_1}$.
   Here $*$ denotes any sequence of sets.
   For instance, if $d_1$ duplicates a diagonal from $S_1$, we match $(S_1,\{d_2\}, \{d_1\},S_2,S_3)$  and $(S_1,\{d_2 \cup d_1\},S_2,S_3)$
   but do not match $(S_1,\{d_1\},\{d_2\},S_2,S_3)$  and $(S_1 \cup \{d_1\},d_2,S_2,S_3)$.

   After Step 1 comes Step 2,  Step 3, etc. The defining rule is:
   \item Step $p$. We match   $(*,S,\{d_p\},*)$ with $(*,S\cup \{d_p\},*)$  if

   \begin{enumerate}
     \item $S\neq S_{j_p}$, and
     \item These two cells are not matched on steps $1,2,...,p-1$.
   \end{enumerate}
 \end{itemize}

 According to R. Forman's theory, all the cells that are matched can be contracted. The above described discrete Morse function orders the movable diagonals.
 For the remaining cells all the movable diagonals appear in singletons and in the chosen order.

 For instance one could have $(S_1,\{d_1\},\{d_2\},S_2,\{d_3\},S_3)$  (provided that $d_1$ and $d_2$ duplicate some diagonals from $S_1$, and
 $d_3$ duplicates a diagonal from $S_2$). Another example: one  has neither $(S_1,\{d_2\},\{d_1\},S_2,\{d_3\},S_3)$  nor  $(S_1,\{d_2,d_1\},S_2,\{d_3\},S_3)$  since these both of them are matched.

 So the cells of the new combinatorial structure on the fragment  are determined by the first two items only:  (1) placing boundary components next to the chosen germ (either to the left or to the right),  (2) duplicating some of the diagonals containing the chosen germs.
 This cell structure equals the join of a number of balls and elementary exceptional fragments.
 Since we have seen that the latter are also balls, the claim is proven.
\qed

\section{Combinatorial formula for the Chern class of a tautological bundle}\label{SecChern}
Section \ref{SecContrBound} provides a  combinatorial model for  the tautological $S^1$-bundle:  we have  triangulated base and triangulated total space of the bundle such that the projection is a simplicial map.
{Thus the local combinatorial formulae
for the first Chern class and its powers (see  Igusa \cite{Igusa} or Mnev and Sharygin \cite{Mnev}) are applicable.}

Let us start with  some auxiliary constructions.
An  \textit{ oriented necklace } (or a \textit{necklace}, for short)  on letters $1,...,k+1$ is an orbit of a word (on the same letters)
under cyclic permutations. One thinks of a necklace as of  a number of beads colored by numbers $1,...,k+1$ on an oriented cyclic thread.

Assume that $k+1$ is odd, and a necklace $\nu$ is fixed.
Let $N_{odd}(\nu)$  (respectively, $N_{even}(\nu)$) be the number of ways to choose exactly $k$ beads of the necklace $\nu$, one bead out of each of the colors, in such a way that the resulted permutation of the chosen beads is odd (respectively, even).\footnote{Although we have a cyclic permutation, its parity is well-defined since $k+1$ is odd. It does not depend on the way one cuts the circle to get a string.}
Set  $$p(\nu)=N_{even}(\nu)-N_{odd}(\nu) ,$$   and set $N_i(\nu)$ to be the  number of beads colored by $i$.

A $k$-dimensional  simplex  $\sigma=\sigma^k$ in $\mathcal{BD}_{g, b-1, n+1; n_2,... ,n_b}$  is labeled by some $(S_1,...,S_{k+1})$.
Therefore the germs of edges emanating from  the first free marked point  $v_1$  have associated numbers, and thus give a necklace $\nu(\sigma)$ on letters $1,...,k+1$.  Although  some of the colors might be missing in a particular necklace  $\nu(v_1,\sigma)$, the color ''$1$'' is always present.

\begin{prop}
  \begin{enumerate}
    \item The cochain
$$Ch(\sigma^2)=\frac{-p(\nu(v_1,\sigma))}{2N_1(N_1+N_2)(N_1+N_2+N_3)}, $$

where $N_i=N_i(\nu(v_1,\sigma))$,
represents the first Chern class of the circle bundle
  $$\pi: \mathcal{BD}_{g, b , n; 1,n_2,n_3,...,n_b} \rightarrow \mathcal{BD}_{g, b-1, n+1; n_2,... ,n_b},$$
  which exgibits a combinatorial model for the tautological bundle $L_1$.
\item
In the same notation, the cochain
$$Ch^h(\sigma^{2h})=\frac{(-1)^h h! \cdot p(\nu(v_1,\sigma))}{(2h)!\cdot N_1(N_1+N_2)(N_1+N_2+N_3)...(N_1+N_2+...N_{2h+1})}$$
represents the $h$-th power of the first Chern class.

\end{enumerate}

\end{prop}

Proof.
A \textit{fattening } $F(\nu)$ of a necklace $\nu$ is a  new necklace obtained from $\nu$  by replacing each bead
''$i$'' by  the cluster of beads \newline  ''$k+1,k,...,i+1,i,i,i+1,...,k,k+1$''. \ \  We call it the\textit{ cluster associated with} $i$.

We shall prove the claim (1); then the proof of (2) is routine.
So assume that a two-dimensional simplex $\sigma^2=(S_1,S_2,S_3)$ is fixed in the base.
There is a natural ordering on its vertices:  $S_1, \ S_1\cup S_2,\ S_1\cup S_2\cup S_3$.
To apply the formula from \cite{Mnev}, one should look at  three-dimensional simplices in its preimage. Each of them is obtained
by duplicating one of the emanating germs.
For instance, duplication of a germ labeled by $1$, one gets $(S_1,\{d\},S_2,S_3)$, $(S_1,S_2,\{d\},S_3)$ , $(S_1,S_2,S_3,\{d\})$.
This sequence of $3$-simplices in the preimage yields sequence of beads $1,2,3$ in the associated necklace.  Since the new diagonal can be added either to the left or to the right of the old one, we have  a cluster of beads $3,2,1,1,2,3$.
One concludes that we arrive at the fattening of the necklace $\nu= \nu(v_1,\sigma)$.
According to the Mnev-Sharygin formula  \cite{Mnev}, we need to compute $p(F(\nu))$, and $N_i(F(\nu))$.
Clearly,

$$N_1(F(\nu))=2N_1(\nu), \ N_2(F(\nu))=2N_1(\nu)+2N_2(\nu),$$$$\ N_3(F(\nu))=2N_1(\nu)+2N_2(\nu) +2N_3(\nu).$$

Once we prove that $$p(F(\nu))=p(\nu)\cdot 2^3,\ \ \ \ \ (*)$$
the statement of the proposition follows.

The proof is based on two observations:

(1) When counting $p(F(\nu))$ one may choose beads from different clusters only.  Indeed, the choices of (at least) two beads from one and the same cluster can be grouped into collections  such that the contribution to $p$ of a collection vanishes.

Examples:  (a) $...,3,\textbf{2},\textbf{1},1,2,3,...$  is grouped with $...,3,{2},\textbf{1},1,\textbf{2},3,...$

(b) $...,3,\textbf{2},{1},1,2,\textbf{3},...$, \ \ \ \  $...,3,{2},{1},1,\textbf{2},\textbf{3},...$\ \ \ \ \newline $...,\textbf{3},\textbf{2},{1},1,2,{3},...$, and  $...,\textbf{3},{2},{1},1,\textbf{2},{3},...$  are grouped.

(2)  When counting $p(F(\nu))$ one may choose beads from different clusters associated to all different letters $1,2,3$ only.
Indeed, other choices can be grouped into mutually cancelling collections.
Therefore, from a cluster associated with $i$  we take one of the two beads $i$. This proves $(*)$.\qed

\bigskip

{If $b=1$, the above theorem gives a formula for the Chern class in the classical setting. We remind the reader that M. Kontsevich gave an expression for it in terms of a differential $2$-form, see \cite{Kon}. }

\end{document}